\documentclass[preprint,12pt]{elsarticle}

\usepackage{graphicx}
\usepackage{epsfig}

\usepackage{amssymb}
\usepackage{amsthm}

\usepackage{lineno}

\usepackage{mathptmx}
\usepackage{latexsym}
\usepackage{paralist}
\usepackage{amsmath,amsbsy,amssymb,color} 
\usepackage{url}
\usepackage{subfigure}

\newdefinition{rmk}{Remark}
\newtheorem{prop}{Proposition}

\journal{
}

\begin{document}

\begin{frontmatter}



\title{Oscillating delayed feedback control schemes for
stabilizing equilibrium points.}


\author{Ver\'onica E. Pastor \& Graciela A. Gonz\'alez}


\begin{abstract}
Limitations of the delayed feedback control and of its extended versions have been
fully treated in the literature. The oscillating delayed feedback
control appears as a promising scheme to overcome this
problem. In this work, two methods based on oscillatory delayed
feedback control schemes for the continuous time case are dealt
with. For both of them, stabilization of an equilibrium point in the
general non-linear scalar case is rigorously proven. Additionally,
their control performance and stability parameters region are
respectively studied.
\end{abstract}

\begin{keyword}
oscillating feedback control \sep delay \sep stability region \sep control performance \sep rate of convergence.
MSC[2010] 34H15 \sep 93D15 \sep 34D20
\end{keyword}

\end{frontmatter}


\section{Introduction}
\label{intro}

It is well known that delayed feedback control (DFC) was originally proposed
by Pyragas in \cite{Pyragas1} for stabilizing an unstable periodic orbit (UPO) in a chaotic system. 
Its most important feature is that it does not
require the exact location of the UPO to
be stabilized. It makes use of a control signal obtained from the
difference between the current state of the system and the state of
the system delayed by the period of the UPO. The DFC method is also reformulated as a tool to stabilize 
equilibria embedded in chaotic attractors (see \cite{survey16} and references within it).
With this objective, it is implemented on known chaotic systems as Chen system \cite{apChen} or
Rossler system \cite{apRossler} and on technical applications like \cite{Leit09} or \cite{Yang15} among others. An extended version (EDFC), proposed in \cite{Socolar}, results more effective for stabilizing highly-unstable equilibrium points and UPO's (\cite{Pyragas2}).\\
It is important to point out that not all UPO's  can be stabilized by time delayed feedback control methods.
Namely, for non-autonomous systems, it is not possible to stabilize a hyperbolic periodic orbit which has an 
odd number of real Floquet multipliers larger than unity. This is known as the odd number limitation (ONL) and it is
stated in \cite{Naka} for DFC and in \cite{Nakajima} for EDFC. The proofs of \cite{Naka} and \cite{Nakajima} do not apply to UPO's 
in the autonomous case (the technical reason is clearly explained in \cite{Hooton}).
Instead, there is a limitation and it also involves the number of Floquet multipliers greater than
unity but in addition, it depends on an analytical expression given by an integral of the control force along the UPO
to be stabilized. This limitation is proven in \cite{Hooton} for DFC and in \cite{Amman} for EDFC. For
equilibrium point stabilization, the ONL does hold true in both autonomous and non-autonomous systems. 
In particular, for the autonomous case, if the linearization matrix has an odd number of positive eigenvalues then 
stabilization is impossible by means of DFC methods (\cite{Pyragas2}, \cite{Kokame}). An interesting review on 
the evolution of the ONL problem and its derivations may be found in \cite{survey16}.\\
Another drawback of time delayed feedback is that the controlled system comes out a delayed 
differential equation, the state space of which is infinite dimensional and hence it is quite difficult
 to state analytical results and to get effective stabilization criteria. Some approaches focussed on overcoming these 
difficulties are based on periodic gain modulation (\cite{Leonov14}) or ``act-and-wait" concept introduced by Insperger (\cite{Insperger} and previous papers of this author). These methods are caracterized by alternaly applying and cutting off the controller in finite intervals yielding to a finite-sized monodry matrix of the closed system so the linear stability of the UPO may be enhanced by an appropiate choosing of the control parameters. Act-and-wait approach has been used together with DFC for stabilizing unstable equilibrium points (\cite{Konishi}), for stabilizing UPO's of nonautonomous systems (\cite{Pyragas4}) and, of autonomous systems (\cite{Pyragas5},\cite{Cetinkaya}).\\       
For stabilizing equilibrium points, a delayed feedback controller is derived in \cite{Kokame} that overcomes the drawbacks of DFC, providing a systematic procedure of its design. However this procedure is valid only for sufficiently short delay time which results inappropiate in certain experimental setting (i.e., in fast dynamical systems due to the finite operating speed of these electronic devices). Later, Konishi \textit{et.al.} (\cite{Konishi}) proposed a DFC based on the 
``act-and-wait" control, the advantage of it being that the controlled system with delay can be described by a discrete-time system without delay. This method works for long delay time and deadbeat controller may be designed by a simple systematic procedure but they can not show that their method overcomes the ONL property.\\
Interestingly, there is  an early contribution for improving the delayed feedback limitations (\cite{Schuster}). The key of
this strategy is to avoid a too rapid decay of the control magnitude
and this could be achieved by applying feedback control only
periodically. In the second part of the paper \cite{Schuster}, the idea of an
oscillating delayed feedback control is translated to the
differential equations. Different from the discrete time case (\cite{Schuster}, \cite{Morgul}), if the oscillating perturbation term involves the difference between current state and delay state, stabilization can not be achieved. Hence an oscillatory velocity
term is introduced in \cite{Schuster}: it is worked out for equilibrium point  stabilization of a scalar linear differential equation, with a rather uncompleted proof and as pointed out in \cite{Pyragas3}, the related stabilizing
result is not clear.\\
This issue is revisited in this work considering the general scalar non-linear case:
\begin{equation} \label{1}     
 \dot{x}=f(x)
\end{equation}
with $x^{*}$ being an unstable equilibrium point of (\ref{1}). Let
us assume that $f$ is continuously differentiable and $f'(x^{*})\!=\!\lambda\!>\!0$.\\
Two oscillating delayed feedback control (ODFC) schemes for equilibrium point
stabilization will be deeply studied. Preliminary ideas on them have been 
introduced by us in \cite{Dincon}. The first method is based on the delayed velocity term  taken from
\cite{Schuster}. In the second one, the perturbation depends on
the difference between two delayed states. In spite of being inspired in \cite{Schuster} both methods may be framed 
within the ``act-and-wait" concept. In fact as the resulting differential equations are affected by delayed feedback only periodically,
a continuously differentiable map is associated to the controlled dynamics and stability will be derived using linearization classical tools. 
Both algorithms will be fully presented and conditions for stabilization will be deduced. Under the stated conditions, the control objective achievement will be rigorously proven for the general nonlinear case. An analytical description of the stability parameters region will be given. Rate of convergence, control performance and stability parameters region of each method will be studied and confronted.

\section{ODFC method based on delayed velocity term}
\label{sec:2}

This control strategy consists in adding a perturbation based on a
delayed velocity term:
\begin{equation} \label{dos2}    \dot{x}(t)=f(x(t))+ \epsilon(t)  \dot{x}(t-\tau)
\end{equation}
where, \begin{equation*}
\epsilon(t)=\left\{\!\begin{array}{cll}
0 & \quad {\rm if} & \quad  2k\tau \leq t < (2k+1)\tau, \\
\epsilon & \quad {\rm if} &  (2k+1)\tau \leq t < (2k+2)\tau,
\end{array}\right . \ \ \ {\rm for} \ \  k \in \mathbb{N}\cup \{0\}
\end{equation*}
being $\epsilon$ and $\tau$ control design parameters.\\
Let us note that $x^{*}$ is preserved as an equilibrium point of
system (\ref{dos2}). Then, $\epsilon$ and $\tau$ for which the system
stabilizes in $x^{*}$ should be found.\\
Activated control depends on $\dot{x}(t-\tau)=f(\varphi(t-\tau))$,
being $\varphi$ the solution of (\ref{dos2}) in the previous
time-interval. Then, system (\ref{dos2}) becomes a non-autonomous
dynamical system described by a smooth piece-wise ordinary
differential equation.
It will be proved that for a certain range of $\epsilon$, depending
on $\lambda$ and $\tau$, $x^{*}$ results an asymptotically stable
equilibrium point. Therefore, if this strategy is applied with
initial condition in a neighborhood of the origin, the control
objective is fulfilled.\\
Putting $\delta x= x-x^{*}$, and $g(\delta x)=f(x^{*}+\delta x)$,
system (\ref{1}) yields to: $\delta \dot{x}=g(\delta x)$ with $g(0)=0$ 
and $g'(0)=\lambda$, while (\ref{dos2}) becomes:
$\delta \dot{x}=g(\delta x)+\epsilon (t) \delta
\dot{x}(t-\tau)$, so without lost of generality, we can assume $x^{*}=0$ and $f'(0)=\lambda$.\\

\begin{rmk}\label{R1}
 Fixed $\epsilon$ and $\tau$, system (\ref{dos2})
is determined by:\\
\begin{equation*}\begin{cases} \dot{x}=F(t,x)=F_k(x,t) \\
x(0)=x_{0} \end{cases}\end{equation*}\\
where for each $k\geq 0$  and 
for $t \in [2k\tau,(2k+2)\tau)$,
\begin{equation*} \label{dos-bis}
F_{k}(t,x)\!=\! \begin{cases} f(x)\! & \mbox{if } 2k\tau \! \leq \! t \!< \! (2k+1)\tau \\
f(x)\!+\!\epsilon \psi_{k}(t) \! & \! \mbox{if }\! (2k+1)\tau \! \leq \! t< \!(2k+2)\tau \end{cases} 
\end{equation*}
and $\psi_k(t)=  \dot{x}(t-\tau)$.\\
Let us note that: (i) $x \equiv 0 $ is solution in $[2k\tau,\!
(2k\!+\!2)\tau)$ of $\dot{x}=F_{k}(t,x) $ with $x(2k\tau)=0$, and (ii)
$\psi_{k}(t)=f(\varphi(t-\tau))$ being $\varphi$ the solution in the
sub-interval $[2k\tau, (2k+1)\tau)$, so $\psi_{k}(t)$ is continuous
on $[(2k+1)\tau, (2k+2)\tau)$. Then, by continuous dependence on
initial condition \cite{KHALIL}, given $\Delta_{k} > 0$, $\exists
\delta_{k} >0$ ($\delta_{k}=\delta_{k}(\epsilon,\tau)$) such that if
$|x_{2k}|< \delta_{k}$, there is a unique solution $x(t)$ of
$\dot{x}=F_{k}(t,x)$ with $x(2k\tau)=x_{2k}$ and $|x(t)|<
\Delta_{k}$ in $[2k\tau,(2k+2)\tau)$. Moreover, $x(t)$ is continuous
in \\$[2k\tau,(2k+2)\tau)$.\\
\end{rmk}

\begin{rmk}\label{R2} The solution of system (\ref{dos2}) on each
interval $[2k\tau,(2k\!+\!2)\!\tau\!)$ is only determined by the $x_{2k}$
value (but it does not depend on $k$). This is a consequence of the
fact that when the control is not active, the system is autonomous
and that when the control is active, it is non-autonomous but its
dependence on $t$ holds on the
solution of the first half of the interval.\\
\end{rmk}

\begin{prop} \label{prop1}
         Let $f \in C^{1}(\mathbb{R})$ with $f(x^{*})=0$ and $f'(x^{*})=\lambda > 0.$
If the parameters $\epsilon$ and $\tau$ verify:
     \begin{equation}\label{4}
         \frac{-2\cosh(\lambda\tau)}{\lambda\tau}<\epsilon<\frac{-2\sinh(\lambda\tau)}{\lambda\tau}.
     \end{equation}
then, $x^{*}$ is an asymptotically stable equilibrium point of the controlled system (\ref{dos2}).
\end{prop}

\begin{proof}
 Let us assume $x^{*}=0.$\\
Let us fix $k\geq 0$ and take $|x_{2k}|$ small enough. As stated
in Remark \ref{R1}, there exists $x(t)$ unique continuous solution of
(\ref{dos2}) in $[2k\tau,(2k+2)\tau)$ with initial condition
$x(2k\tau)=x_{2k}$.\\
From Remark \ref{R2}, the map $P$ determined by 
$  x_{2k+2}=P(x_{2k})$
being
     \begin{equation} \label{3bis}
         x_{2k+2}=\displaystyle\lim_{t \to (2k+2)\tau^{-}} x(t)\,
     \end{equation}
is well defined. Map $P$ results from the composition of $p$ and
$\widetilde{p}$ given by: 
\begin{center}$p:  x_{2k+1}=x((2k+1)\tau)=p(x_{2k})$ \ \ \ $\widetilde{p}: x_{2k+2}=\widetilde{p}(x_{2k+1})$\end{center}
Note that $x^{*}=0$ is fixed point of $P$ and
$P'(0)=\widetilde{p'}(0)p'(0)$.\\
Let $\phi(t,x_{2k})$ the solution of $\dot{x}=f(x)$ with initial
condition $x(2k\tau)=x_{2k}$ in \\ $[2k\tau,(2k+1)\tau)$. This solution
satisfies:
     \begin{equation}
            \phi(t,x_{2k})= x_{2k}+\int\limits_{2k\tau}^t f(\phi(s,x_{2k}))ds.  
     \end{equation}
As $f$ is $\mathcal{C}^{1}$, by differentiation under the integral
sign, an expression of \\$\displaystyle\frac{\partial\phi}{\partial x_{2k}}((2k+1)\tau,x_{2k})$ is deduced and it results:
     \begin{equation}
            \frac{\partial\phi}{\partial x_{2k}}((2k+1)\tau,0)= e^{\lambda\tau}.
     \end{equation}
Since $p(x_{2k})=\phi((2k+1)\tau,x_{2k})$, then $p'(0)=e^{\lambda\tau}$.\\
Besides, being $\phi(t,x_{2k+1})$ the solution of
$\dot{x}=f(x)\!+\!\epsilon \dot{x}(t\!-\!\tau)$ in $[(\!2k\!+\!1)\tau,(\!2k\!+\!2\!)\tau\!)$ 
with initial condition $x(\!(2k\!+\!1)\tau)=x_{2k+1}$, it satisfies:
     \begin{equation}
            \phi(\!t,\!x_{2k\!+\!1})=
            x_{2k\!+\!1}+\int\limits^t_{(\!2k\!+\!1)\tau}\![f(\phi(\!s,x_{2k\!+\!1}))\!+\!\epsilon\dot{\phi}(\!s\!-\!\tau,x_{2k})]ds.
     \end{equation}
Analogously to the first part, $\displaystyle\frac{\partial\phi}{\partial x_{2k+1}}((2k+2)\tau,x_{2k+1})$ is obtained and
     \begin{equation}
            \frac{\partial \phi}{\partial x_{2k+1}}((2k+2)\tau,0)= e^{\lambda\tau}+\epsilon\lambda\tau.
     \end{equation}
As $\tilde{p}(x_{2k+1})=\phi((2k+2)\tau,x_{2k+1})$, then $\tilde{p}'(0)=e^{\lambda\tau}+\epsilon\lambda\tau.$
\\
Therefore, \vspace{-0.2cm}
     \begin{equation}\label{3}
         P'(0)=e^{\lambda\tau}(e^{\lambda\tau}+\epsilon\lambda\tau)
     \end{equation}
\vspace{-0.2cm} which is of modulus less than 1 iff $\epsilon$ and $\tau$ verify (\ref{4}).\\
\\
As $f$ is $\mathcal{C}^1$, the continuous differentiability
of solution $\phi(t,\cdot)$ on initial conditions is argued
(\cite{PERKO}). Then, $p$ and $\tilde{p}$ are $\mathcal{C}^1$, and
therefore, $P$ is $\mathcal{C}^1$, too.\\
Let us introduce $\alpha\!=\!P'(0).$ As $P(0)=0$, it follows from the mean value theorem that
$P(x)=P'(\delta)x$ for some $\delta$ between 0 and $x$. As $P'$ is continuous and $|P'(0)|=|\alpha|<1$,
fixed $\widetilde{\alpha}:|\alpha|<|\widetilde{\alpha}|<1$, $|P(x)|<|\widetilde{\alpha}||x|<|x|, \forall x: |x|<\widetilde{\delta}$, for $\widetilde{\delta}$ sufficiently small.\\
\smallskip
Hence, if $\epsilon$ and $\tau$ verify (\ref{4}) there exists
$\widetilde{\delta}$ such that if $\!|x_0|\!<\!\widetilde{\delta}$:
\vspace{-0.2cm}
\begin{equation}\label{eq} 
|x_{2k+2}|<|x_{2k}| \; \ \ \forall k \geq 0
\end{equation}
Let us fix $\Delta>0$ and $\delta=min\{\tilde{\delta}, \delta_0\}$
where $\delta_0$ is as in Remark 1 for $\Delta_0 = \Delta$. Then,
taking, $|x(0)|<\delta$, it results $|x_{2k}|<\delta$ for all
$k\geq 0$, that together with (\ref{3bis}) yields to the
existence of a unique (continuous) solution of (\ref{dos2}) for all
$t \geq 0$. Moreover, $|x(t)|< \Delta, \forall t \geq 0$
and, the stability of the origin is shown. In turn, $|P'(0)|<1$ also
yields to $\displaystyle\lim_{k \to\infty} x_{2k}=0$, which
implies $\displaystyle\lim_{k \to\infty} x_{2k+1}=0$ and it
results $\displaystyle\lim_{t \to\infty} x(t)=0$ so asymptotic stability is obtained.\\
\end{proof}
\begin{rmk}\label{R3}
For each $(\!\epsilon,\!\tau)$ verifying
(\ref{4}), there exists $\!\alpha\!\in\! (\!-\!1,\!1)$:
\vspace{-0.2cm}
\begin{equation} \label{eq11}  
 \epsilon = \frac{e^{-\lambda\tau}(\alpha-e^{2\lambda\tau})}{\lambda\tau}
\end{equation}
and viceversa. Indeed $\alpha=P'(0)$.
\end{rmk}
Particulary, if the function is linear, that is, $f(x)= \lambda x$,
the map $P$ is also linear, namely, $P(x)=\alpha x$ and the
incidence of $\alpha$ on the convergence speed is evident. For
example, let us take $f(x)=2x$: the trajectories resulting from applying (\ref{dos2}) with $\tau =
0.2$ and $x_{0}=0.5$, $\alpha=-0.4$ and $\alpha=0.8$ are confronted
in Figure \ref{fig:1}.

\begin{figure*}
 \includegraphics[width=16cm,height=8cm]{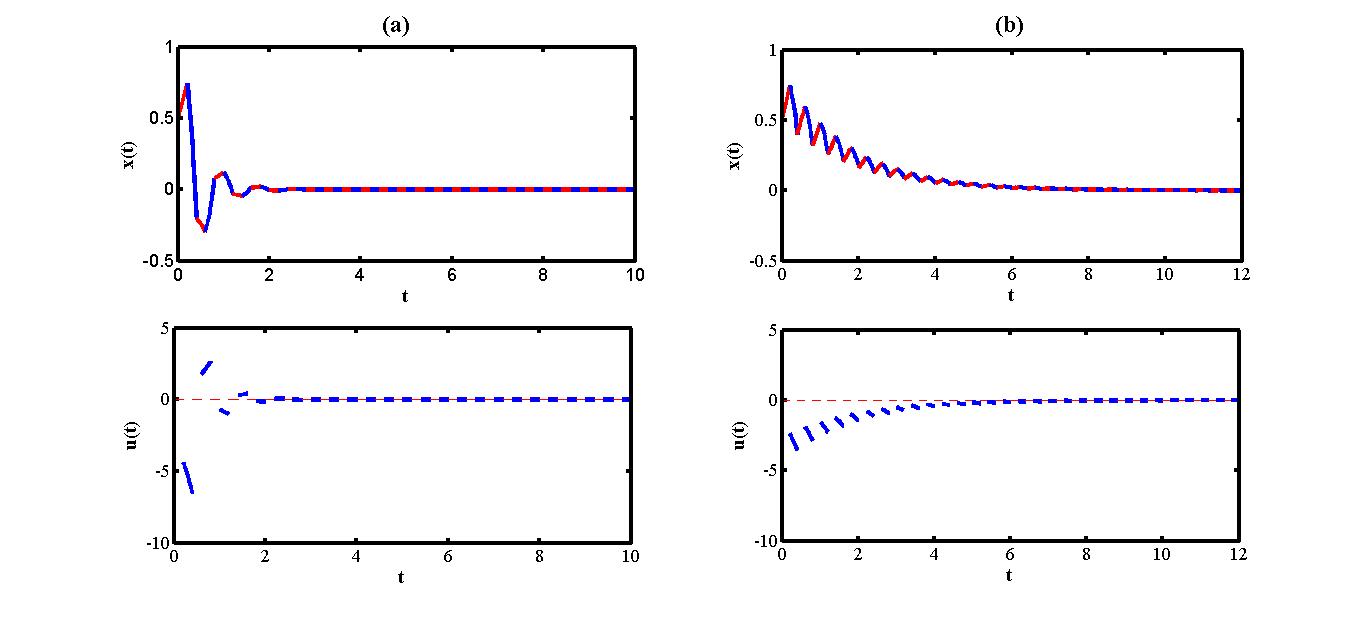}
\caption{State behavior and control performance of system
           (\ref{dos2}) with $f(x)=2x, x_{0}=0.5, \tau=0.2$ (a) $\alpha =-0.4$ (b) $\alpha =0.8$.}
\label{fig:1}       
\end{figure*}

On the other hand, the control parameter $\tau$ also affects the
convergence of the system trajectories: indeed as $\tau$ is smaller,
faster convergence comes out. Taking again $f(x)=2x$ and $\alpha=
-0.4, 0.8$, and changing $\tau$ by 0.4, speed of convergence is
slower (Figure \ref{fig:2}) than in the respective first examples (Figure
\ref{fig:1}).
\begin{figure*}
  \includegraphics[width=16cm,height=8cm]{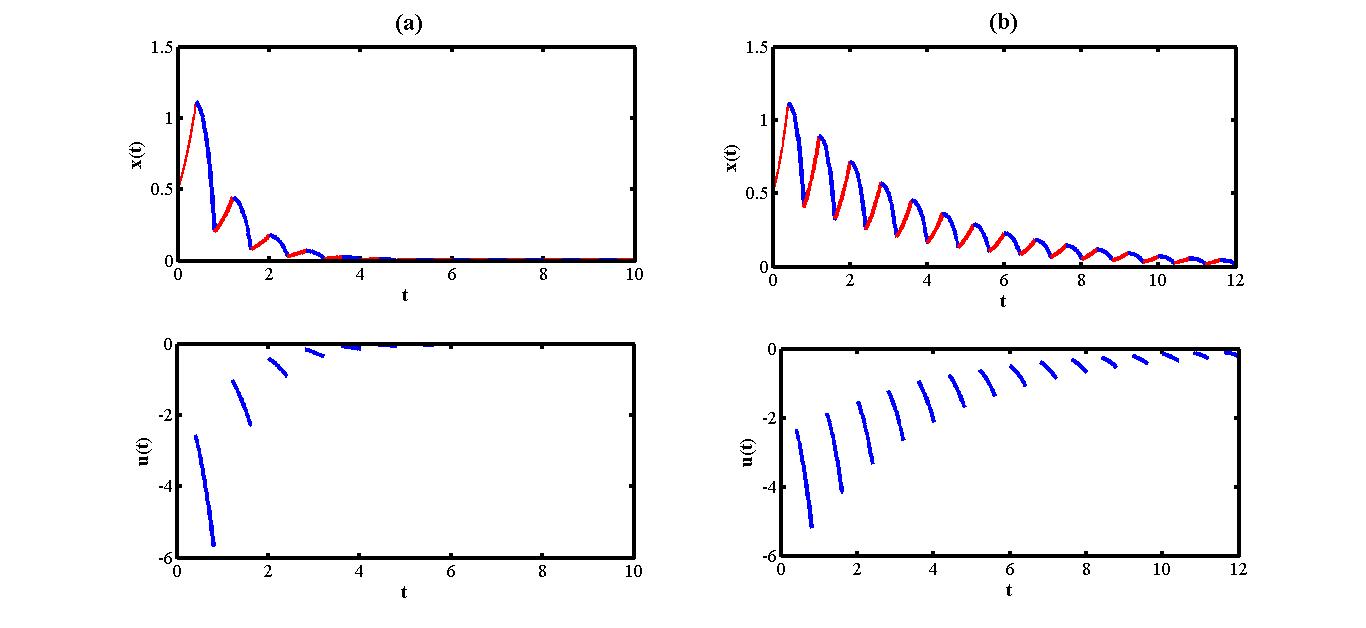}
\caption{State behavior and control performance of system (\ref{dos2}) with $f(x)=2x, x_{0}=0.5, \tau=0.4$. (a) $\alpha =-0.4$ (b) $\alpha =0.8$.}
\label{fig:2}       
\end{figure*}
\\
Due to the theoretical result, it is known that state signal convergence is achieved if the initial condition is taken near enough to the equilibrium point. For different nonlinear cases, simulations show how these effects on the state signal are inherited from the linearized system, although compensated by a rise in control magnitude.
This is illustrated for different nonlinear functions, with
$\alpha=-0.4$, $\tau=0.2$ and $x_{0}=0.5$ in Figure \ref{fig:3}
(confront to Figure \ref{fig:1}(a)).\\
\begin{figure*}
  \includegraphics[width=16cm,height=8cm]{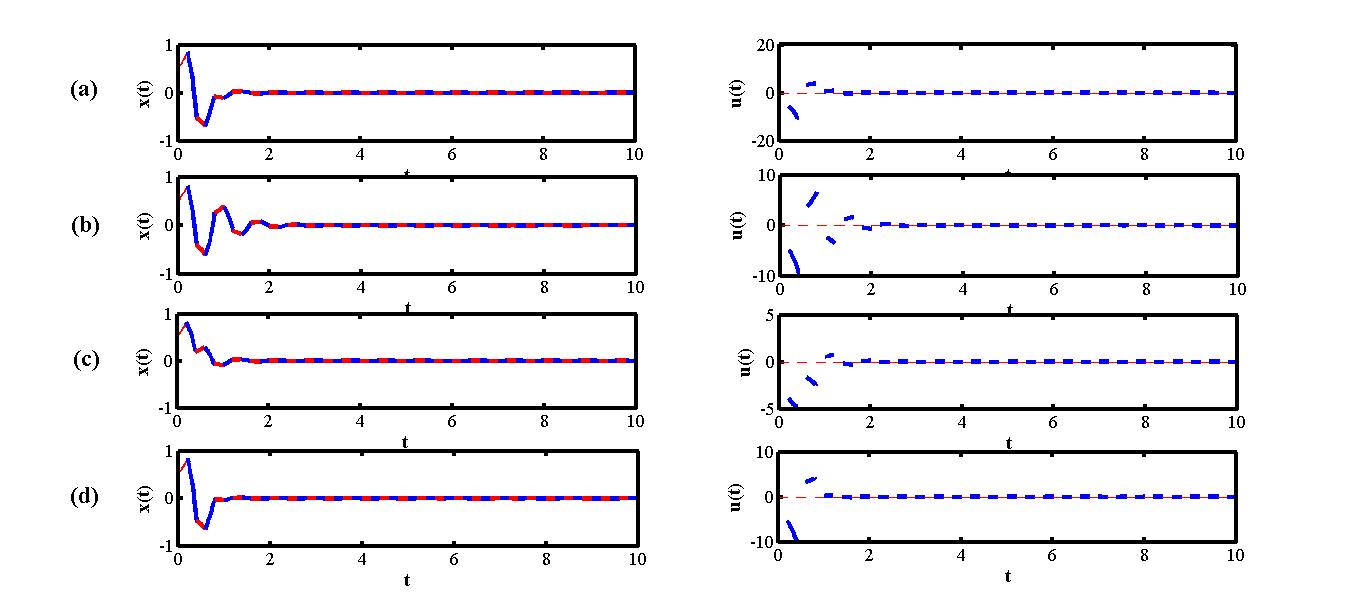}
\caption{State behavior and control performance of system (\ref{dos2}) for $x_{0}=0.5, \tau=0.2,
\alpha=-0.4$. (a) $f(x)=2x+x^{2}$, (b) $f(x)=2x+x^{3}$, (c) $f(x)=2x-x^{3}$, (d) $f(x)=2x+sin^{2}(x)$.}
\label{fig:3}       
\end{figure*}
\\
Signal exponential convergence is also revealed. The exponential decay curves envolving 
the signal as displayed in Figure \ref{fig:4} put  this feature even in more evidence. 
In fact, fixed $\epsilon$ and $\tau$, being $x(t)$ the solution of ($\ref{dos2}$) and $\alpha$ determined by ($\ref{eq11}$)
it is not difficult to prove that given a small $\mu > 0, \exists \delta_\mu$ such that if $|x_0| \leq \delta_\mu$:
\vspace{0.21cm}
\begin{equation} \label{eq12} 
 \begin{split}
& c_m e^{{\huge \frac{\ln(|\alpha|-\mu)}{2\tau}t}}|x_0| \leq |x(t)| \leq  c_M e^{\frac{\ln(|\alpha|+\mu)}{2\tau}t}|x_0|  \quad \text{if }  \alpha \neq  0 
 \end{split}
\end{equation}
and,
\begin{equation*}
 \begin{split}
& 0 \leq |x(t)| \leq  c_M e^{\frac{\ln \mu}{2\tau}t}|x_0|  \quad \text{if }  \alpha =  0
 \end{split}
\end{equation*}
for certain positive constants $c_m, c_M$. For linear systems, (\ref{eq12}) is also valid with $\mu=0$. Figure \ref{fig:4}(a) 
represents the upper inequality in one of these cases. This inequality may be verified even in the nonlinear case by taking $|x_0|$ small enough (Figure \ref{fig:4}(b)).\\
\begin{figure*}
  \includegraphics[width=16cm,height=7cm]{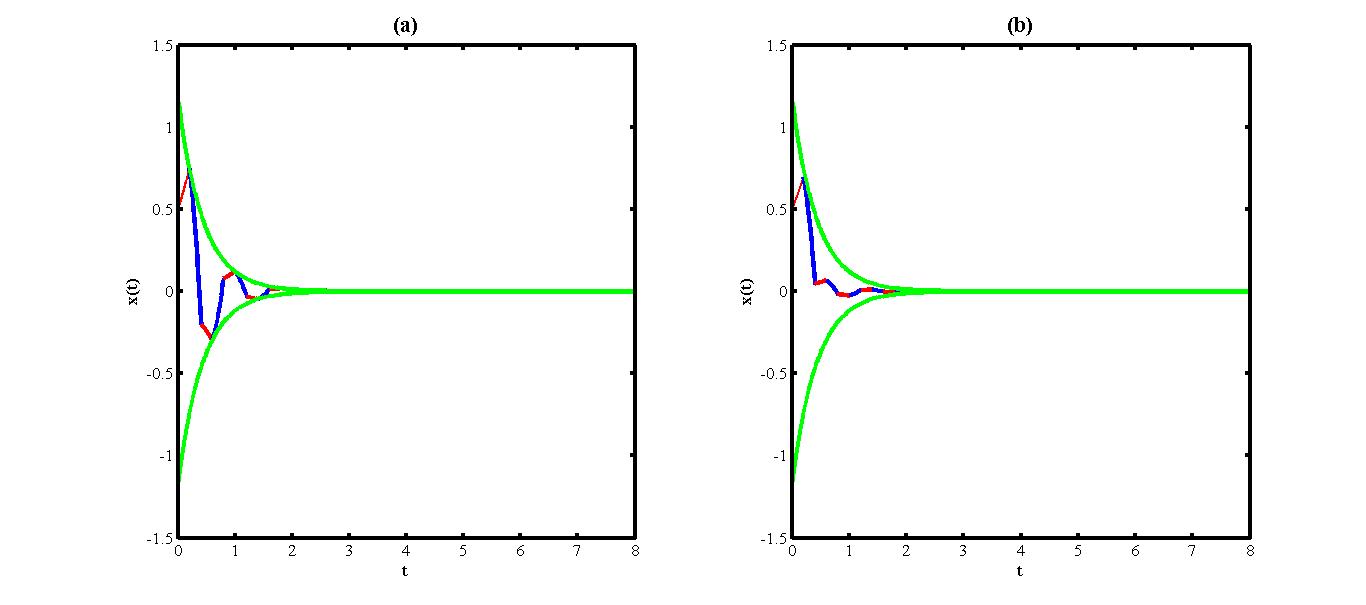}
\caption{Exponential stability for system (\ref{dos2}) with 
$ x_{0}=0.5, \tau=0.2$ and $\alpha =-0.4$: (a)  $f(x)=2x$; (b) $f(x)=2x-x^3$.}
\label{fig:4}       
\end{figure*}
\\
Hence, a convergence rate $\beta$ of algorithm (\ref{dos2}) may be stated as:
\begin{equation} \beta = \begin{cases} \frac{\ln|\alpha|}{2\tau}, \quad \text{if }  \alpha \neq  0 \\
-\infty, \quad \text{if }  \alpha =  0
\end{cases}
\end{equation}
\\
Although the rate of convergence is optimized
fixing $\alpha$ equal zero and $\tau$ as small as possible, choosing $\tau$ too small 
makes control magnitude take very large values during transitory. For example, influence of $\tau$-value on
trajectory behavior and on control cost resulting from applying
the method to $f(x)=2x$, $x_{0}=0.5$ and $\alpha=0$ with
$\tau=0.2$ and $\tau=0.02$ is illustrated in Figure \ref{fig:5}. Namely, the scale
change is fully appreciated by confronting control signal of Figure \ref{fig:5}(a) and Figure
\ref{fig:5}(b).\\
\begin{figure*}
  \includegraphics[width=16cm,height=8cm]{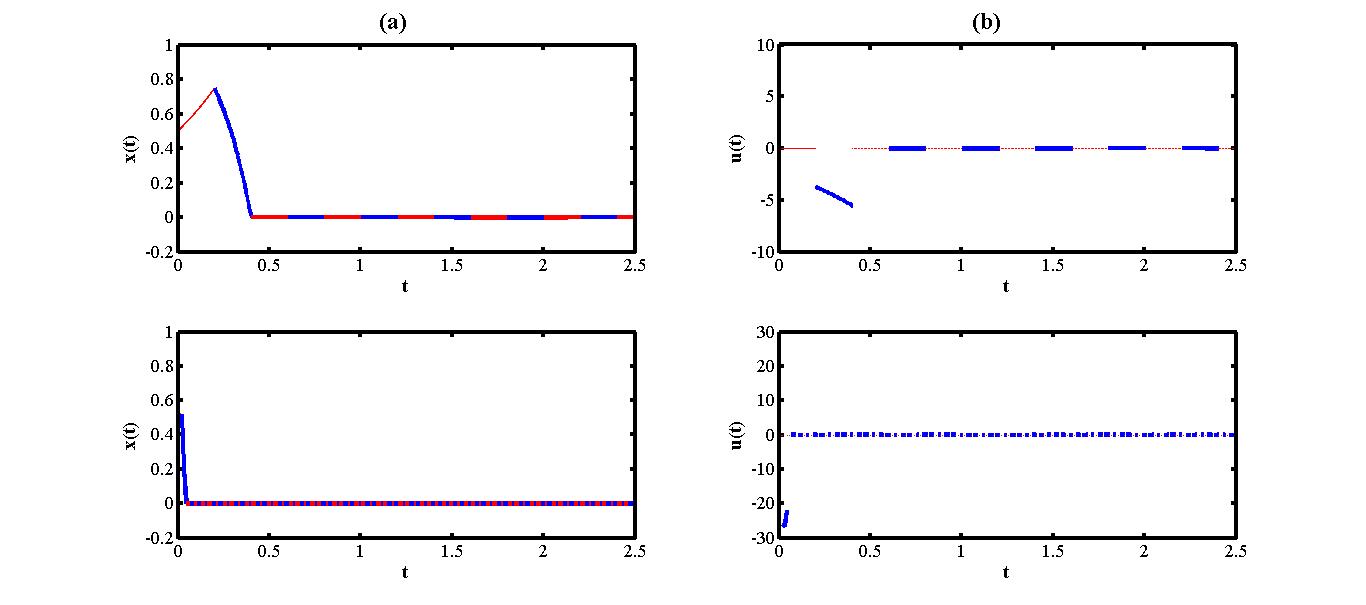}
\caption{State behavior and control performance of system (\ref{dos2}) with $f(x)=2x$,
           $x_{0}=0.5$, $\alpha= 0$ (a) $\tau=0.2$ (b) $\tau=0.02$.}
\label{fig:5}       
\end{figure*}
\\
This phenomena is better understood by paying attention to stability parameters region, i.e.
the region of the control parameter values for which the stability
objective is achieved. In Figure \ref{fig:6}(a), the stability
parameters region - which is obtained from (\ref{4}) - is
illustrated. The lower and upper bounds of $|\epsilon|$ are the curves defined by $\alpha = 1$ and $\alpha = -1$, respectively (Figure \ref{fig:6}(c)).
Note that if $\tau$ is near zero, for any $\alpha$,
there is a dramatic increase of $|\epsilon|$ (Figure
\ref{fig:6}(b)), so affecting the control performance. However, it is proved analytically that for a fixed $\alpha$, 
there exists a unique $\tau$ that minimizes the absolute value of the control gain; 
namely $\displaystyle\tau^{*}=\frac{1}{\lambda} \Big(1-\frac{2\alpha}{\alpha+ e^{2\lambda\tau^*}}\Big)$. Hence the choosing of adequate $\alpha$ and $\tau$ depends on
a compromise between rate of convergence and control magnitude.

\begin{figure*}
\includegraphics[width=16cm,height=8cm]{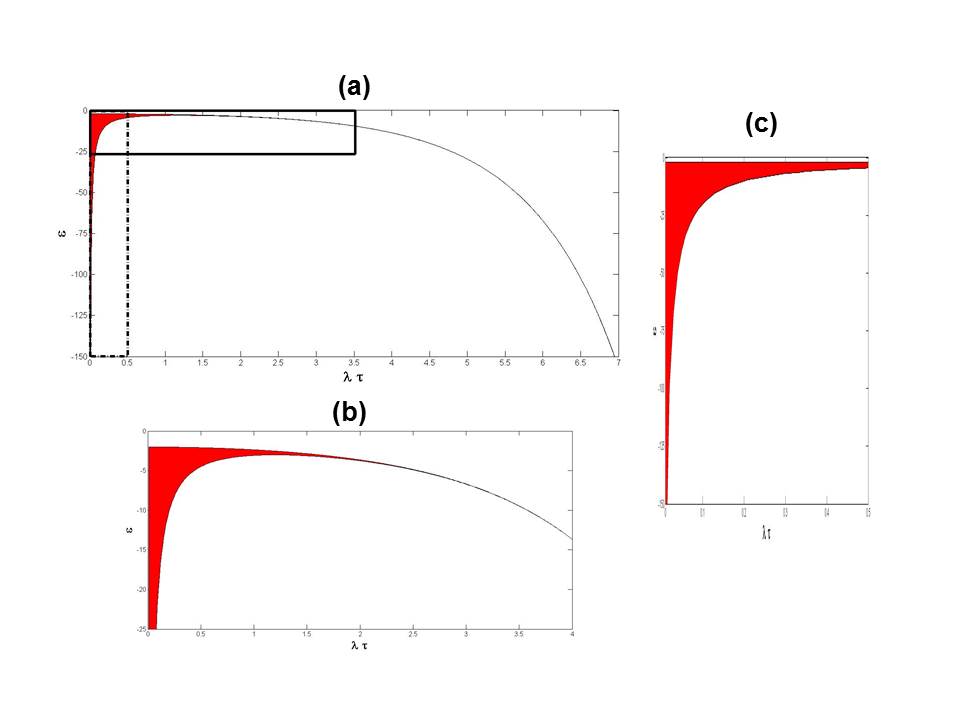}
\caption{(a) Stability parameters region of (\ref{dos2}). (b) Zoom in for  $-25 < \epsilon \leq  0$. (c) Zoom in for $0 < \lambda \tau \leq  0.5$.}
\label{fig:6}
\end{figure*}

\section{ODFC method based on delayed states difference}
\label{sec:3}

It is easy to verify in the scalar case, that if the oscillating perturbation involves the difference between current state and delayed state (and even for the generalized version as proposed in \cite{Konishi}) stabilization can not achieved by any control parameters. 
In this proposal, the difference between two delayed states is introduced into the perturbation:
\begin{equation}\label{7-3}   \dot{x}(t)=f(x(t))+
\epsilon(t)(x(t-2\tau)-x(t-\tau))
\end{equation}
where \begin{center} $\epsilon(t) = \begin{cases} 0, & \mbox{if } 3k\tau\! \leq \!t <\! (3k\!+\!2)\tau \\
\epsilon, & \mbox{if } (3k\!+\!2)\tau\! \leq \!t<\!(3k\!+\!3)\tau
\end{cases}$ \ \ \ for $k\! \in \mathbb{N}\! \cup \! \{\!0\}.$\end{center}
As in the first method, $x^{*}$ is preserved as an equilibrium point
and without lost of generality, we assume $x^{*}=0$. System (\ref{7-3})
also comes out a non-autonomous dynamical smooth piece-wise differential
equation and it is also possible to state a range of $\epsilon$,
depending on $\lambda$ and $\tau$ such that if  this strategy is
applied with initial condition in a neighborhood of the origin, the
control objective is fulfilled. The proof follows similar steps to
the stabilization proof of the first method.\\
\begin{prop} \label{prop2}
Let $f \in C^{1}(\mathbb{R})$ with $f(x^{*})=0$ and $f'(x^{*})=\lambda > 0.$
If the parameters $\epsilon$ and $\tau$ verify:
     \begin{equation}\label{pp2}
         \frac{e^{3\lambda\tau}-1}{\tau e^{\lambda\tau}(e^{\lambda\tau}-1)}<\epsilon< \frac{e^{3\lambda\tau}+1}{\tau e^{\lambda\tau}(e^{\lambda\tau}-1)}
     \end{equation}
then, $x^{*}$ is an asymptotically stable equilibrium point of the controlled system (\ref{7-3}).\end{prop}
\begin{proof}
Let us assume $x^{*}=0$. Existence, unicity and continuity of the solutions in $[3k\tau,(3k+3)\tau)$ for all
$k\geq 0$, result as in Proposition 1.\\
Here, the map $P$ defined by $x_{3k+3}=\displaystyle\lim_{t \to (3k+3)\tau^{-}} x(t)\,=P(x_{3k})$ for $k\geq
0$, has $x^{*}=0$ as fixed point and $P'(0)=\widetilde{p}'(0)p'(0)$ with: 
\begin{equation*} p: x_{3k+2}=x((3k+2)\tau)=p(x_{3k})\end{equation*}
and, \begin{equation*}\widetilde{p}:x_{3k+3}=\widetilde{p}(x_{3k+2}).\end{equation*}
Let $\phi(t,x_{3k})$  the solution of (\ref{7-3}) in $[3k,(3k+2)\tau)$  with initial condition\\ $x((3k)\tau)=x_{3k}$.
By using the integral formulation, as $f$ is $C^1$, it results:
     \begin{equation*}
         p'(0)=\frac{\partial\phi}{\partial x_3k}((3k+2)\tau,0)=e^{2\lambda\tau}
     \end{equation*}
Idem, for $\phi(t,x_{3k+2})$, the solution of (\ref{7-3}) with initial condition\\ $x((3k+2)\tau)=x_{3k+2}$, it is obtained:
     \begin{equation*}\label{8-0}
         \widetilde{p}(0)=\frac{\partial\phi}{\partial x_3k+2}((3k+3)\tau,0)=e^{\lambda\tau}+\epsilon\tau(1-e^{\lambda\tau})e^{-\lambda\tau}.
     \end{equation*}
Therefore, 
     \begin{equation}\label{8}
         P'(0)=e^{3\lambda\tau}[1+\epsilon\tau(1-e^{\lambda\tau})e^{-2\lambda\tau}]
     \end{equation}
which is of modulus less than 1 iff $\epsilon$ and $\tau$ verify (\ref{pp2}).\\
As in Proposition 1, it is shown that if $\epsilon$ and $\tau$
verify (\ref{prop2}); there exists $\widetilde{\delta}$ such that if $|x_0| < \widetilde{\delta}:$
\begin{equation}\label{eq2} |x_{3k+3}|<|x_{3k}| \; \ \ \forall k \geq 0.
\end{equation}
 The existence of a unique continuous solution of (\ref{7-3}) for all
$t\geq 0$ is stated by following the same
technical considerations as in Proposition 1. In turn, this yields to the asymptotic stability of the origin.
\end{proof}
%
\begin{rmk}\label{R4}
 As for the first method, introducing $\alpha=P'(0) \in (-1,1) $
the relationship (\ref{pp2}) may be formulated through:\\
\begin{equation} \label{eq18}   \epsilon =  \frac{e^{-\lambda\tau}(e^{3\lambda\tau}-\alpha)}{\tau(e^{\lambda\tau}-1)}
\end{equation}
for $\alpha$: $|\alpha|<1$.
\end{rmk}
Comments about the control performance of this method are quite
similar to the ones on the first method. For illustration see
Figures \ref{fig:7}, \ref{fig:8}, \ref{fig:9} and \ref{fig:10}.
\begin{figure*}
  \includegraphics[width=16cm,height=8cm]{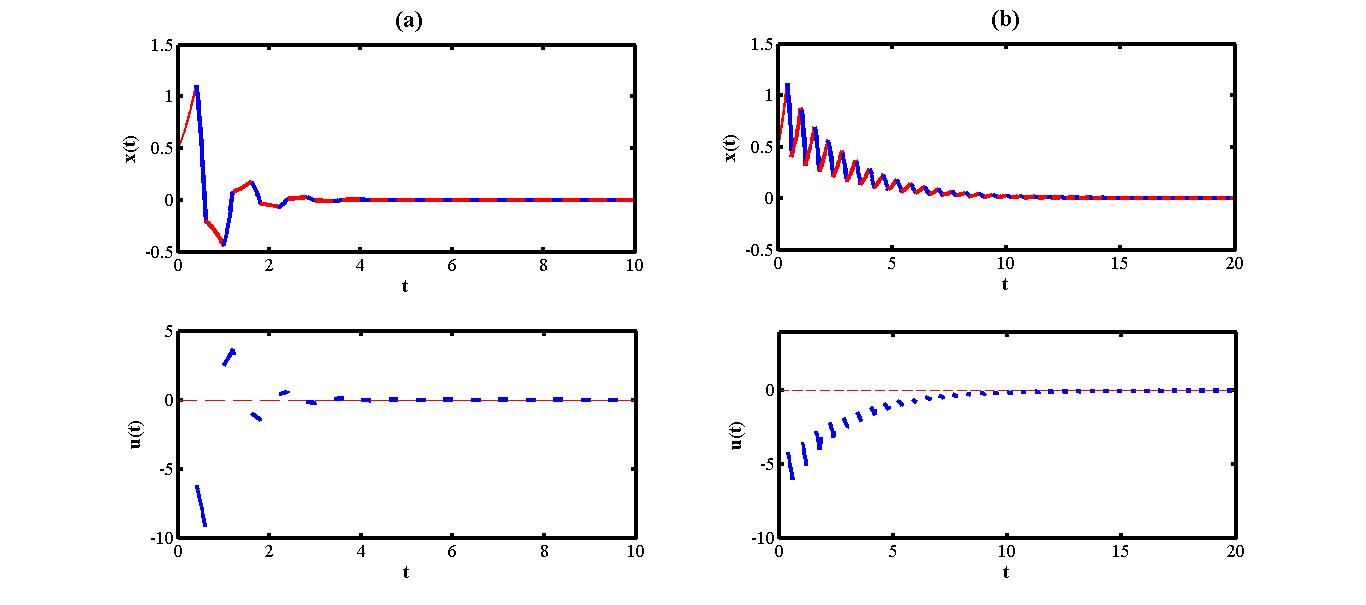}
\caption{State behavior and control performance of system
           (\ref{7-3}) with $f(x)=2x, x_{0}=0.5, \tau=0.2$ (a) $\alpha =-0.4$ (b) $\alpha =0.8$.}
\label{fig:7}       
\end{figure*}
\\
\begin{figure*}
  \includegraphics[width=16cm,height=8cm]{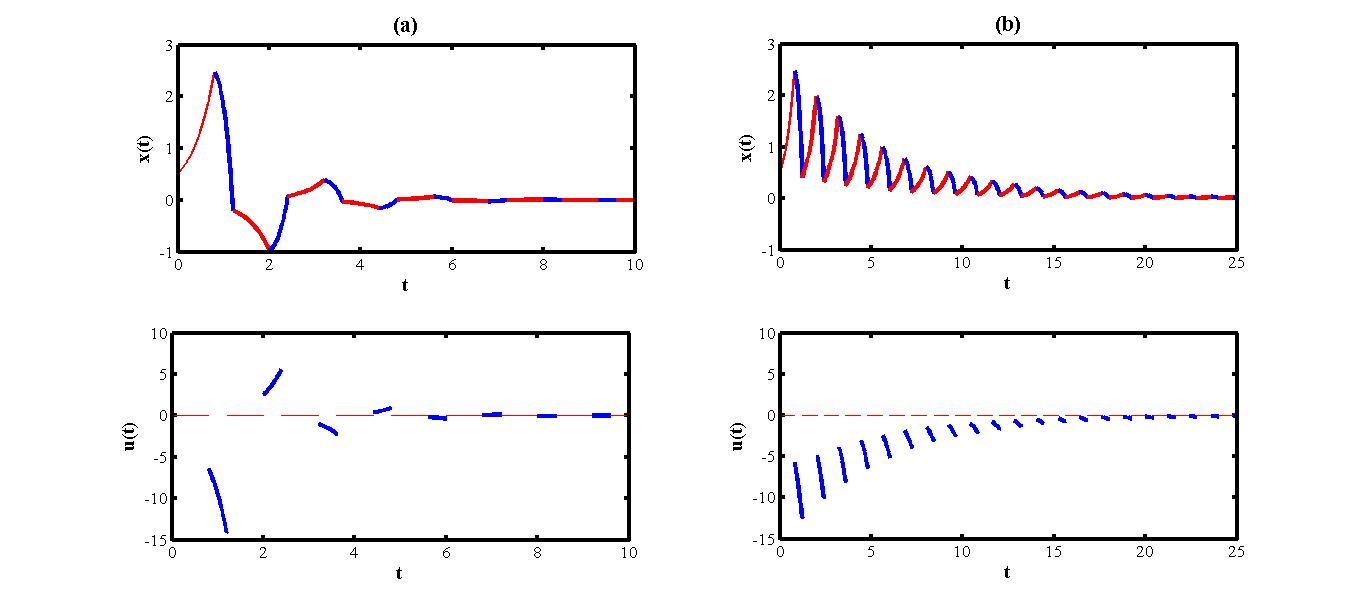}
\caption{State behavior and control performance of system (\ref{7-3}) with $f(x)=2x, x_{0}=0.5, \tau=0.4$. (a) $\alpha =-0.4$ (b) $\alpha =0.8$.}
\label{fig:8}       
\end{figure*}
\\
\begin{figure*}
  \includegraphics[width=16cm,height=8cm]{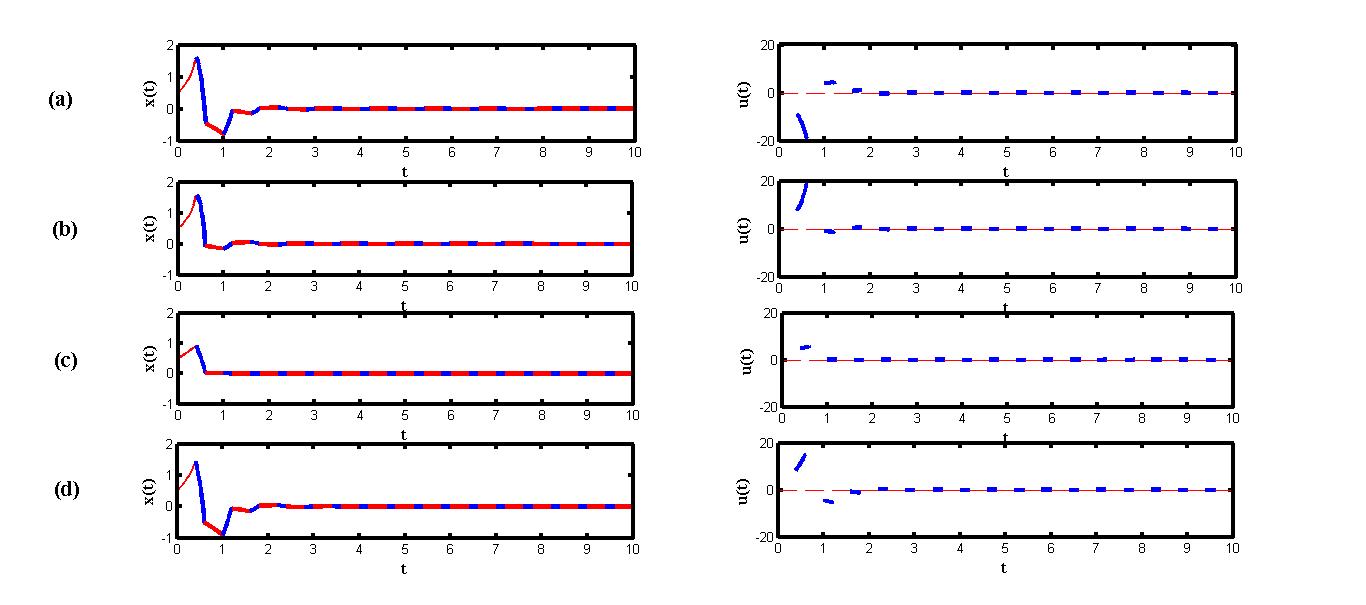}
\caption{State behavior and control
performance of system (\ref{7-3}) for $x_{0}=0.5, \tau=0.2, \alpha=-0.4$. (a)$f(x)=2x+x^{2}$, (b)$f(x)=2x+x^{3}$,
(c)$f(x)=2x-x^{3}$, (d)$f(x)=2x+sin^{2}(x)$.}
\label{fig:9}       
\end{figure*}
\\
\begin{figure*}
  \includegraphics[width=16cm,height=7cm]{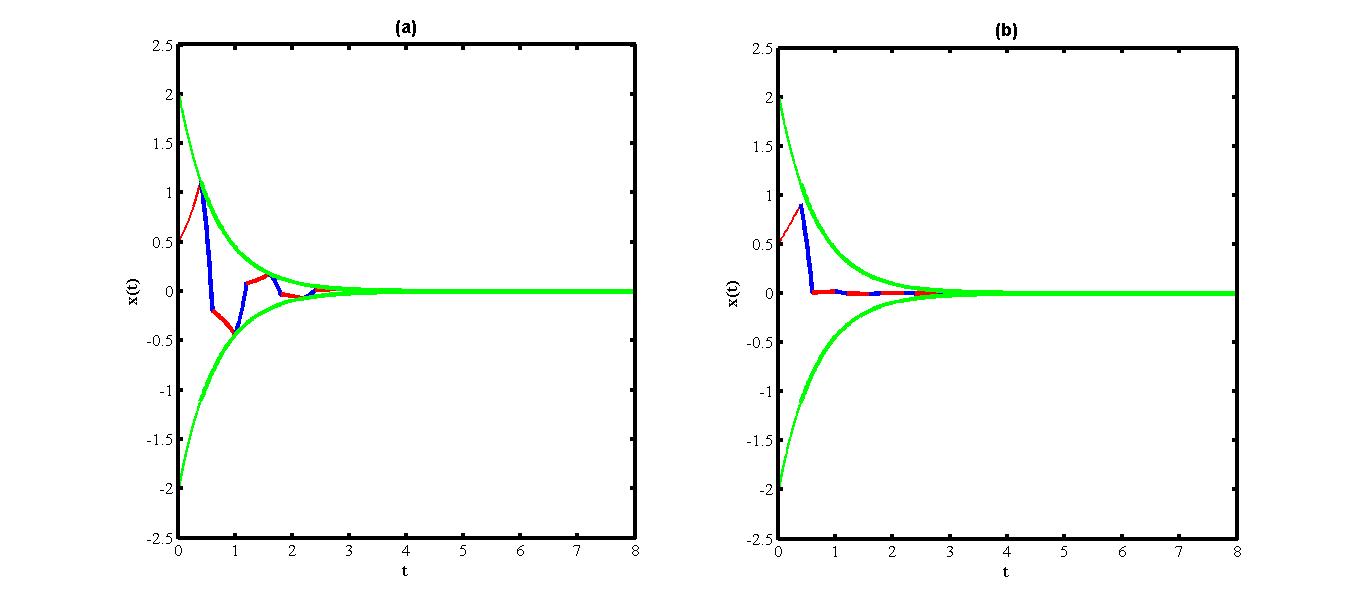}
\caption{Exponential stability for system (\ref{7-3}) with 
$ x_{0}=0.5, \tau=0.2$ and $\alpha =-0.4$: (a)  $f(x)=2x$; (b) $f(x)=2x-x^3$.}
\label{fig:10}       
\end{figure*}
The exponential decayment is also valid in this case:
\begin{equation*}
 \begin{split}
& c_m e^{\frac{\ln(|\alpha|-\mu)}{3\tau}t} |x_0| \leq |x(t)| \leq c_M e^{\frac{\ln(|\alpha|+\mu)}{3\tau}t} |x_0| \quad \text{if }  \alpha \neq 0 \\
\text{and,}\\
& 0 \leq |x(t)| \leq c_M e^{\frac{\ln \mu}{3\tau}t} |x_0| \quad \text{if }  \alpha =  0  
 \end{split}
\end{equation*}\\
for certain positive constants $c_m$, $c_M$.\\
And the convergence rate $\beta$ of algorithm (\ref{7-3}) comes out:\\
\begin{equation*}  \begin{split}
\beta = \begin{cases} \frac{\ln|\alpha|}{3\tau}, & \mbox{if }  \alpha \neq  0 \\
-\infty, & \mbox{if }  \alpha =  0 \end{cases}
 \end{split}
\end{equation*}
Equation (\ref{pp2}) states the stability parameters region of this method. Graphically, it is displayed in Figure \ref{fig:11}.
Consideration about the choosing of the design control parameters are similar to the ones of the first introduced method. In particular, it is convenient to choose $\tau$ near $\tau ^*$, the minimazing value of $\epsilon/ \lambda$, which for a fixed $\alpha$, is given by: 
\begin{equation*} \tau^{*}= \frac{1}{\lambda} \left( \frac{(e^{2 \lambda \tau^{*}} - \alpha e^{-\lambda \tau^{*}})(e^{\lambda \tau^{*}}-1)}{e^{3 \lambda \tau^{*}} - 2 e^{2 \lambda\tau^{*}} - \alpha e^{-\lambda\tau^{*}}+ 2 \alpha} \right) \end{equation*}
%
\begin{figure*}
\includegraphics[width=16cm,height=8cm]{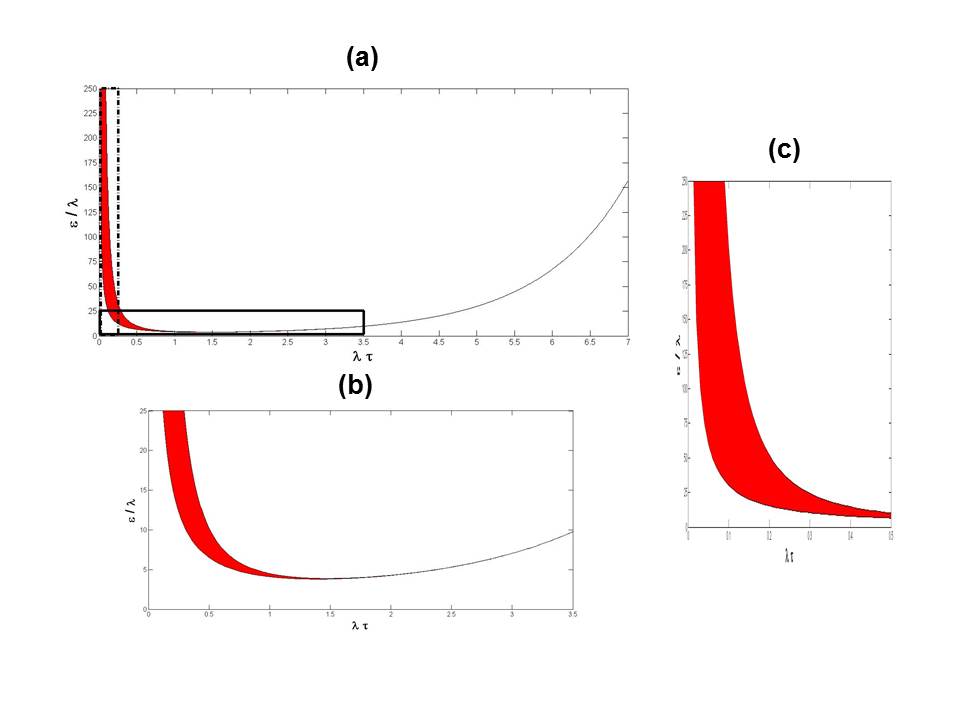}
\caption{(a) Stability parameters region of (\ref{7-3}). (b) Zoom in for  $0 < \epsilon/\lambda \leq  25$. (c) Zoom in for $0 < \lambda \tau \leq  0.5$.}
\label{fig:11}       
\end{figure*}

\section{Concluding remarks and future research}

Two  methods based on ODFC schemes for the continuous time case has been dealt with.
The first one coincidences with the proposal of \cite{Schuster}, based on a delayed velocity term, but extended to the general nonlinear case. For the second method,  the unsuccessful perturbation that depends on only one delayed state, is replaced by one involving two-delayed states. The methodology developed to prove the achievements of the first strategy has been straightforward transferred to prove analogous features on the second one. Hence, for both of them, local stabilization of an equilibrium point in the general non-linear scalar  case has been rigorously proven. The key ingredient of this proof is the bulding of a discrete-time map which reflects the dynamics of the controlled system.
Let us emphasize that the controlled system is a discontinuous time-delayed system but the associated discrete-time system is described by a $C^1$ map so stability is obtained from its linearization which can be computed for any nonlinear system. Then, from continuous dependence on initial conditions, the stabilization of the continuous time system comes out. Additionally, the stability parameter region is explicitly described and in particular, the adequate parameters for deadbeat control  ($\alpha = 0$)
are easily obtained. \\
A wide simulation work let us claim that the first method displays better control performance features than the second one.
This may even be appreciated by confronting the few examples of Section 2 with the respective examples of Section 3. 
Namely, from obtaining the exponential bound of the solution, a quantification for the rate of convergence was stated. 
This index of convergence and a detailed analysis of the stability parameters region confirm the claimed conjectures.\\
These strategies may be developed to stabilize equilibrium points in the n-dimen-\\sional case under adequated observability and controlability conditions without presenting the restrictions of the DFC methods studied in \cite{Kokame} and \cite{Leonov14-2}. More interestingly a right extension of our second method appears as a candidate for overcoming the ONL, coming out as an alternative of \cite{Konishi} in which there is also an ``on-off switching" feedback gain but it does not work in the one dimensional case.\\
As the second method avoids the computation of the derivative, its numerical implementation may result more efficient just because it is not desirable to produce derivative signal $\dot{x}(t)$ from noisy measurements of $x(t)$.
Namely, its extension to the stabilization of UPO is quite simple. Suppose that $\widetilde{x}(t)$ is a UPO
 and its period $T$ is known. By introducing $\delta x=x-\widetilde{x}(t)$, the oscillating feedback control based on delayed states becomes:\begin{equation*}
         u(\!t)\!=\!K(\!t)[\delta x(\!t\!-\!2T)\!-\!\delta x(\!t\!-\!T)]\!=\! K(\!t)[x(\!t\!-\!2T)\!-\!x(\!t\!-\!T)]
\end{equation*}
being $K(\!t\!)$ the oscillating control gain. As in Pyragas method, it
does not require the exact location of the UPO to be stabilized. So stated, it appears as an alternative to the proposals in \cite{Leonov14}, \cite{Pyragas5} and \cite{Cetinkaya}. The problem of UPO stabilization yields to the problem of stabilizing the origin in the non-autonomous n-dimensional case. Note that the kind of periodicity that define $K(t)$ is quite similar to the switch on and off of the ``act-and-wait-time-delayed" feedback control used in these works so the extension of our scheme to UPO stabilization could contribute to advance on these issues. This problem, and additionally, its application for controlling chaos,
i.e., for the UPO embedded in a strange attractor, is part of our
future research. 

\section{Acknowledgements}

This work was supported by UBACyT 2014-2017 (20020130200093 BA GEF).
 


\end{document}